\documentclass{amsart}

\usepackage{graphicx}

\newtheorem{theorem}{Theorem}[section]
\newtheorem{corollary}{Corollary}[theorem]
\newtheorem{lemma}[theorem]{Lemma}

\theoremstyle{definition}

\theoremstyle{remark}

\numberwithin{equation}{section}

\begin{document}
\

\title{Protected Vertices in Motzkin trees}

\author{Anthony Van Duzer}
\address{Department of Mathematics, University of Florida, Gainesville, Florida 32601}
\curraddr{Department of Mathematics,
University of Florida, Gainesville, Florida 32611}
\email{avanduzer@ufl.edu}




\begin{abstract}
In this paper we find recurrence relations for the asymptotic probability a vertex is $k$ protected in all Motzkin trees.  We use a similar technique to calculate the probabilities for balanced vertices of rank $k$.  From this we calculate upper and lower bounds for the probability a vertex is balanced and upper and lower bounds for the expected rank of balanced vertices.
\end{abstract}

\maketitle

\section{Introduction}

\subsection{Motzkin Tree}
A Motzkin tree, also referred to as a 0-1-2 tree, or unary-binary tree, is a rooted plane tree where each vertex may have either 0, 1, or 2 children.
\begin{theorem}
The generating function, $M(x)=\displaystyle\sum m_n x^n$, for the number of all Motzkin trees with n vertices is given by $M(x)=\dfrac{1-x-\sqrt{1-2x-3x^2}}{2x}$ .
\end{theorem}
\begin{proof}
This comes from the relationship $M(x)=x+xM(x)+xM(x)^2$ which is based on the fact the root can be a leaf, the parent of a single child, or the parent of two children and each child of the root would be another Motzkin tree.  Given that relationship you can use the quadratic formula to arrive at the desired generating function.
\end{proof}
\subsection{$k$-protected}  A vertex is said to be of rank $k$ if the shortest path, travelling strictly from parent to child, from the vertex to a leaf is of length $k$.  A vertex is $k$ protected if it is of rank $j$ for some $j \geq k$.  So a leaf would be 0 protected and rank 0, the parent of a leaf is rank 1 and both 0 and 1 protected, and the parent of children who are all 1 protected is 2 protected, but not necessarily rank 2.  One of the advantages of working with $k$-protected over rank k is the recurrent nature of protection; a vertex is $k$ protected if and only if all it's children are $k-1$ protected the same statement does not hold true for rank $k$ since a vertex could be rank 6 and have a child of rank 7.  A vertex is balanced if the shortest path is the same length as the longest path so a balanced $k$-protected vertex would be a vertex where all paths to a leaf are of length at least $k$ and all of these paths are the exact same length and a balanced vertex of rank $k$ is a vertex where all paths to a leaf are of length exactly $k$.
\subsection{Purpose of paper}
In this paper we will find a recurrence relation for the probability a vertex is $k$-protected, we will find a recurrence relation for the probability a vertex is balanced and of rank k, and finally we will get upper and lower bounds for the probability a vertex is balanced and the expected rank of a balanced vertex.

\section{Leaves}The proportion of leaves in Motzkin trees had previously been calculated by Gi-Sang Cheon and Louis W. Shapiro in \cite{A}; however, for the sake of completeness and to give a brief introduction to the technique used in this paper it is included.
\begin{theorem}
The generating function for the number of leaves in all Motzkin trees with  n vertices is given by $L(x)=\dfrac{x}{\sqrt{1-2x-3x^2}}$.
\end {theorem}
\begin{proof}
We observe the relationship $L(x)=x+xL(x)+2xL(x)M(x)$.  To prove this we isolate the root.  If the root is a leaf it contributes x to the generating function.  If the root is not a leaf it either has one child or two children.  If it has a single child then the number of leaves the tree possess is simply the number of leaves of the subtree giving rise to the term $xL(x)$.  If it has 2 children than we can again cut off the root creating two subtrees. On one subtree we count the total number of possible leaves which is given by $L(x)$; we multiply this by the total number of configurations. The total number of configurations with that many leaves is the total number of trees we can make for the right subtree which is simply the number of Motzkin trees hence the term $M(x)$. The power series $xL(x)M(x)$ is the generating function for the total number of leaves only counting the left subtree.  Once we multiply this by 2 we get the total number of leaves in all trees where the root has two children.  Plugging in the known $M(x)$ and doing some algebraic manipulation we arrive at the given generating function.
\end{proof}
\begin {corollary}
The probability that a random vertex of a random Motzkin-tree is a leaf converges to $\frac{1}{3}$ as the number of vertices goes to infinity.
\end{corollary}
\begin{proof}
We can extract the coefficent from the generating function and we get that $l(n)$, the total number of leaves in all Motzkin trees of size $n$, is asymptotically equal to $$\dfrac{\left(\sqrt{\dfrac{3}{\pi}}\right)3^n}{2\sqrt{n}}$$ whereas the number of vertices is asymptotically equal to 
$$\dfrac{n3^{n+1}\sqrt{3}\left(1+\frac{1}{16n}\right)}{(2n+3)\sqrt{(n+2)\pi}}.$$
Comparing the 2 and taking the limit as n goes to $\infty$ we get the ratio is $\frac{1}{3}$.
\end{proof}
This is also the probability a vertex is balanced and of rank 0 since every leaf is balanced.
\section{$k$-protected vertices}
\begin{lemma}
The generating function for the number of vertices that are k protected in all trees with n vertices is given by $$P_k(x)=\dfrac{R_k(x)}{\sqrt{1-2x-3x^2}}=\dfrac{xR_{k-1}(x)+xR_{k-1}^2(x)}{\sqrt{1-2x-3x^2}}$$ where $P_k(x)$ is the generating function for all $k$-protected vertices in all Motzkin tree with n-vertices and $R_k(x)$ is the number of trees where the root is $k$-protected.
\end{lemma}
\begin{proof}
This proof comes in two parts.  The first part is that $P_k(x)=\dfrac{R_k(x)}{\sqrt{1-2x-3x^2}}$ and the second is that $R_k(x)=xR_{k-1}(x)+xR_{k-1}^2(x)$. 
\vspace{.1 in}
\newline
First we will prove that $P_k(x)=\dfrac{R_k(x)}{\sqrt{1-2x-3x^2}}$ . This follows from the same recurrence relationship we used to find the generating function for the number of leaves.  We have $P_k(x)=R_k(x)+xP_k(x)+2xP_k(x)M(x)$  which gives us \newline $P_k(x)=\dfrac{R_k(x)}{\sqrt{1-2x-3x^2}}$ after a bit of manipulation.
\vspace{.1 in}
\newline
Next we need to prove the relationship $R_k(x)=xR_{k-1}(x)+xR_{k-1}^2(x)$.  This comes from the fact the root of a tree can only be k protected is if it has a single  child who is $k-1$ protected or it has two children who are both $k-1$ protected.
\end{proof}
\begin {lemma} (Bender's Lemma) \cite{D}
\newline
Take generating functions $A(x)=\sum a_nx^n$ and $B(x)=\sum b_n x^n$ with radius of convergence $\alpha >\beta \geq 0$ where $\alpha$ goes with $A(x)$ and $\beta$ goes with $B(x)$.  If $\frac{b_{n-1}}{b_n}$ approaches a limit b as n approaches infinity and $A(b) \neq 0$ then $c_n \sim A(b)b_n$ where $\sum c_nx^n$=A(x)B(x).
\end {lemma}
\begin{theorem}
Let $p_k$ be the  asymptotic proportion of all $k$-protected vertices in all Motzkin trees compared to all vertices in all Motzkin trees then $p_k=\frac{1}{3}p_{k-1}+\frac{1}{3}p_{k-1}^2$.
\end{theorem}
\begin{proof}
Consider $R_k$ it is of the form $R_k=\dfrac{A_k+B_k(\sqrt{1-2x-3x^2})}{2x}$ where $A_k$ and $B_k$ are polynomials.  This means that $$R_{k+1}=x\dfrac{A_k^2}{(2x)^2}+x\dfrac{A_k}{2x}+x\dfrac{B_k^2(1-2x-3x^2)}{(2x)^2}+x\left(\dfrac{B_k}{2x}+\dfrac{2A_kB_k}{(2x)^2}\right)\sqrt{1-2x-3x^2}.$$ 
Now we go back to look at $P_k(x)$.  We have $$P_{k+1}(x)=\dfrac{x\dfrac{A_k^2}{(2x)^2}+x\dfrac{A_k}{2x}+x\dfrac{B_k^2(1-2x-3x^2)}{(2x)^2}+x\left(\dfrac{B_k}{2x}+\dfrac{2A_kB_k}{(2x)^2}\right)\sqrt{1-2x-3x^2}}{\sqrt{1-2x-3x^2}}.$$
We can now simplify this and we get 
$$P_{k+1}(x)=\dfrac{x\dfrac{A_k^2}{(2x)^2}}{\sqrt{1-2x-3x^2}}+\dfrac{x\dfrac{A_k}{2x}}{\sqrt{1-2x-3x^2}}+\dfrac{\dfrac{B_k^2(1-2x-3x^2)}{(2x)^2}}{\sqrt{1-2x-3x^2}}+x\left(\dfrac{B_k}{2x}+\dfrac{2A_kB_k}{(2x)^2}\right).$$
Now we look at what each of these contribute to the asymptotic behavior of $P_{k+1}(x)$.  We know that $x\left(\dfrac{B_k}{2x}+\dfrac{2A_kB_k}{(2x)^2}\right)$ is just a polynomial so this summand doesn't contribute anything to the asymptotic behavior.  That leaves us with $$\dfrac{x\dfrac{A_k^2}{(2x)^2}}{\sqrt{1-2x-3x^2}}+\dfrac{x\dfrac{A_k}{2x}}{\sqrt{1-2x-3x^2}}+\dfrac{\dfrac{B_k^2(1-2x-3x^2)}{(2x)^2}}{\sqrt{1-2x-3x^2}}.$$
Given that we can now apply Bender's Lemma with the growth rate of $\dfrac{1}{3}$.  Applying that we see that $\dfrac{\dfrac{B_k^2(1-2x-3x^2)}{(2x)^2}}{\sqrt{1-2x-3x^2}}$ contributes nothing to the asymptotic behavior of $l(n)$ since plugging $\dfrac{1}{3}$ into the numerator we get 0.  So all that is left is $$\dfrac{x\dfrac{A_k^2}{(2x)^2}}{\sqrt{1-2x-3x^2}}+\dfrac{x\dfrac{A_k}{2x}}{\sqrt{1-2x-3x^2}}.$$
This gives us the desired result that $p_{k+1}=\dfrac{1}{3}p_k^2+\dfrac{1}{3}p_k$.
\end{proof}
\begin{corollary}
The growth rate of the asymptotic probabilities that a vertex is $k$ protected is $\frac{1}{3}$.
\end{corollary}
\begin{proof}
From the previous theorem we know that $p_k=\frac{1}{3}p_{k-1}+\frac{1}{3}p_{k-1}^2$. From this we can deduce that the lower growth rate is $\frac{1}{3}$ since the $p_k$ are probabilities and hence non-negative.  For the upper growth rate we observe that the $p_k$ are strictly decreasing and non-negative.  They are strictly decreasing because for every vertex that is $k$-protected there is at least one vertex that is $(k-1)$-protected, its child.  This means that for any $\epsilon$ there exist an m such that for any $l>m, \hspace{.1 in} p_{l} \leq \epsilon.$ Thus for $k \geq m$ we have $p_{k+1} \leq \frac{1}{3}p_k+\frac{1}{3}\epsilon p_k$ that gives an upper growth rate of $\frac{1}{3}+\frac{\epsilon}{3}$ and that goes towards $\frac{1}{3}$ as $\epsilon$ goes to zero.
\end{proof}
\begin{center}
 \begin{tabular}{|c | c|} 
 \hline
 Protection level & Probability $k$-protected $\approx$\\ [.5ex] 
 \hline
 1 & .66666667\\ 
 \hline
 2 & .37037037\\
 \hline
 3 & .16918153\\
 \hline
 4 & .06593464\\
 \hline
 5 & .02342734\\ 
\hline
6 & .007992060\\
[1ex] 
 \hline
\end{tabular}
\end{center}
\section{Balanced Vertexes}
We can use the same technique as in the preceeding to calculate the proportion of balanced vertices of rank $k$.
\begin{lemma}
The generating function for the number of trees whose root is balanced and of rank k, call it $B_k(x)$ is a polynomial.
\end{lemma}
\begin {proof}
If the root is balanced and rank $k$ then all paths to a leaf must be of length $k$.  Thus the largest tree we can have is a full binary tree with $k+1$ levels.  This means we can have at most $2^{k+1}-1$ vertices in our tree.  Thus the largest term that can appear in $B_k(x)$ is an $x^{2^{k+1}-1}$.
\end{proof}
\begin {theorem}
The generating function for the number of vertices who are k protected and balanced is given by $B_k^*(x)=\dfrac{B_k(x)}{\sqrt{1-2x-3x^2}}$.
\end{theorem}
\begin {proof}
The proof is exactly the same as the proof for the number of vertices that are $k$ protected.
\end{proof}
The fact that we are only dealing with polynomials makes balanced much easier to deal with and allows us to calculate rank $k$ rather than $k$ protected.
\begin{theorem}
The asymptotic proportion of balanced vertices of rank $k$ is given by $b_k=\frac{1}{3}b_{k-1}+\frac{1}{3}b_{k-1}^2$.
\end {theorem}
\begin{proof}
The fact that $B_k(x)$ is a polynomial means we can simply plug in the growth rate of $\dfrac{x}{\sqrt{1-2x-3x^2}}$ directly into the polynomial and as we saw before that growth rate is $\dfrac{1}{3}$.  So we plug $\frac{1}{3}$ into $B_k$ and $B_k=xB_{k-1}(x)+xB_{k-1}^2(x)$ which gives the desired recurrence relationship.
\end{proof}
It is natural now to consider the proportion of vertices that are balanced.
\begin{lemma}
 Let $B^*(x)$ equal the generating function for the number of vertices who are balanced. Then $$B^*(x)=\displaystyle\sum_{k \geq 0} B_k^*(x)=\dfrac{\displaystyle\sum_{k \geq 0} B_k(x)}{\sqrt{1-2x-3x^2}}.$$
\end {lemma}
\begin {proof}
This follows from the fact that a vertex is balanced if and only if it is balanced of rank $k$ for some $k$.  We know that $\displaystyle\sum_{k \geq 0} B_k^*(x)$ is well defined as a formal power series since the coefficent of $x^n$ in $B_k^*(x)$ is 0 for $n <k+1$ similarly we also know that $\displaystyle\sum_{k \geq 0} B_k(x)$ is well defined.
\end {proof} 
The form of the power series shows that we would like to use Bender's lemma to get a value for the probability a vertex is balanced so we need to show that $\displaystyle\sum_{k \geq 0} B_k(x)$ is non-zero and has raidus of convergence strictly greater than $\frac{1}{3}$.
\begin{lemma}
Let $b(n)$ be the number of Motzkin trees on n vertices where the root is balanced. Then $b(n)\leq \dfrac{2.9^n}{n^2}$.
\end{lemma}
\begin{proof}
This can be proven by induction and the fact that 
\newline
$b(n)\leq \displaystyle\sum_{k=log_2(n)}^{n-log_2(n)} b(k)b(n-k-1)$.
\end{proof}
From this we know that the exponential growth rate is 2.9 and hence it has a radius of convergence greater than 3.  From this we know that the probability a vertex is balanced converges as n goes to infinity.
\begin{corollary}
The probability that a vertex is balanced is between 
\newline
0.568362259762727779 and  0.5683622597627278
\end{corollary}
\begin{proof}
The probability a vertex is balanced is simply the sum of the probabilities that it is $k$ balanced taken over all $k$.  This is because if a vertex is balanced it must be balanced of some rank $k$.   So we have $P(balanced)=\displaystyle\sum_{k=0}^\infty b_k$.  Now we need to estimate this sum.  We have that $\left(\dfrac{1}{3}\right)^k(b_m) \leq b_{k+m} \leq \left(\dfrac{1}{3}+b_m\right)^k(b_m)$.  We can prove this inductively since we know that $b_{k+1}=\dfrac{1}{3}b_k+\dfrac{1}{3}b_k^2$ which means $\dfrac{1}{3}b_k \leq b_{k+1} = \left(\dfrac{1}{3}+b_k\right)b_k$ we can then continue this getting that 
\newline
$$\left(\dfrac{1}{3}\right)^2 b_k \leq b_{k+2} =  \left(\dfrac{1}{3}+b_{k+1}\right)\left(\dfrac{1}{3}+b_k\right)b_k \leq \left(\dfrac{1}{3}+b_k\right)^2 b_k.$$ This is because $b_k$ is a decreasing sequence.  We can continue this to get the originally stated inequality.  This gives us $$\displaystyle\sum_{k=0}^{20} b_k+\displaystyle\sum_{k=20}^\infty \left(\dfrac{1}{3}\right)^{20-k}b_{20} \leq P(balance) \leq  \displaystyle\sum_{k=0}^{20} b_k+\displaystyle\sum_{k=20}^\infty \left(\dfrac{1}{3}+b_{20}\right)^{20-k}b_{20}.$$
The infinite sum on the far left and far right side of this inequality both are simply geometric series so are easily summable.  This gives us the stated lower and upper bounds.
\end{proof}
\subsection{Expected Value}
We can use a similar technique to calculate the expected rank of balanced vertices.   
\begin{theorem}
The expected value of the rank of balanced vertices exist as n goes to infinity exist.
\end{theorem}
\begin{proof}
Let $EB(x)=\displaystyle\sum_{k=0}^\infty kB_k^*(x)$.  This is well defined as a formal power series for the same reason that $B_k^*(x)$ was well defined.  To apply Bender's lemma to this we need to find the exponential growth rate of the numerator.  Let $eb(n)$ be the coefficent of $x^n$ in the numerator of $EB(x)$ we have that $eb(n) \leq nb^*(n)$ since $x^n$ will not appear in any $B_k^*(x)$ for $k \geq n$  and $b^*(m)=\displaystyle\sum_{k=0}^\infty b_k^*(m).$  From this we know that $eb(n) \leq \dfrac{2.9^n}{n}$, so the numerator of the generating function has exponential growth rate of at most 2.9.  Because the exponential growth rate is less than 3 we know the function has radius of convergence greater than $\dfrac{1}{3}$  which means that Bender's lemma applies. Let $N(x)$ be the numerator of the generating function for $EB(x)$ then asymptotically $EB(x)$ is of the form $N(\frac{1}{3})k_n$ where $k_n$ is the $x^n$ term in $\dfrac{1}{\sqrt{1-2x-3x^2}}$ and since we know the proportion of $k_n$ to all vertices in all Motzkin trees this means that the stated proportion exist as n goes to infinity.
\begin {lemma}
The ratio between the coefficent of $x^n$ in $EB(x)$ and and the number of all vertices in all Motzkin trees divided by the proportion of vertices that are balanced is equal to the expected rank of a balanced vertex in a tree of size n.
\end {lemma}
Combining lemma 4.7 with the proof that the asymptotic proportion of $EB(x)$ compared to the number of vertices in all Motzkin trees exist shows us that the expected value exist as n goes to infinity. 
\end{proof}
\begin{theorem}
The expected rank of a vertex that is balanced is between 
\newline
.6464847301966947 and .64648473019669473
\end {theorem}
\begin{proof}
We use the standard expected value formula $\mathbb{E}(X)=\displaystyle\sum_{j\geq 0} jP(X=j)$.  Getting upper and lower bounds for this is very similar to the technique we used to find the probability a vertex is balanced.  We calculate directly the sum of the first 19 terms and then we use the fact that $\displaystyle\sum_{j\geq 20} jr^{j-20}=\dfrac{20-19r}{(-1+r)^2}$ with r the common ratio so to get a lower bound we use $\frac{1}{3}$ and to get an upper bound we get $\frac{1}{3}+b_{20}$.  Solving this and dividing by the probability a vertex is balanced and we get the stated expectation.
\end{proof}
\section{Open questions}
We were able to prove that the expected value existed for balanced vertices and calculated fairly accurate bounds for it but the question remains can you show that the expected value exist for general vertices nad if so what is the expected value.
\vspace {.1 in}
\newline
The sequence for number of leaves in al Motzkinl trees with $n-1$ steps is a rather interesting sequence in that it is the number of all paths from (0,0) to $(n,n)$ that avoid 3  right steps in a row that starts with a right step and ends with an up step.  This is interesting because there is an easy bijection from the number of Motzkin trees to the number of Dyck paths that avoid 3 right steps.  Is there a similar bijection for leaves.  If so does it translate easily to rooted planar trees with more than 2 children.
\vspace {.1 in}
\newline
Another question would be what proportion of vertices are leaves, or more generally $k$-protected, for rooted planar trees where the out degree of every vertex is between 0 and $n$ and if we take these probabilities as sequences do the sequences converge to the proportion of vertices that are $k$-protected in all rooted planar trees.

\end{document}